\documentclass[a4paper, 10pt]{amsart}

\usepackage{amsthm, amsmath, amssymb}
\usepackage{xcolor}
\usepackage{enumitem}
\usepackage{microtype}
\usepackage[noadjust]{cite}
\usepackage[colorlinks, citecolor=red, urlcolor=blue, bookmarks=false, hypertexnames=true]{hyperref}

\theoremstyle{plain}
\newtheorem{theorem}{Theorem}[section]
\newtheorem*{theorem*}{Theorem \ref{thm:appl}}

\newtheorem{lemma}[theorem]{Lemma}

\newcommand{\ang}[2]{\langle #1, #2 \rangle}

\theoremstyle{definition}

\newtheorem{remark}[theorem]{Remark}

\numberwithin{equation}{section}

\DeclareMathOperator{\trace}{trace}
\DeclareMathOperator{\grad}{grad}

\DeclareMathOperator{\Div}{div}
\DeclareMathOperator{\id}{Id}
\DeclareMathOperator{\Hess}{Hess}

\DeclareMathOperator{\Ric}{Ric}
\allowdisplaybreaks

\title[Rigidity results for biconservative hypersurfaces]{Compact Biconservative Hypersurfaces in Space Forms: Rigidity Without Scalar Curvature Assumptions}
\author{ Aykut Kayhan}
\address{Permanent address: Mathematics and Science Education, Maltepe University, Istanbul, 34480, Turkey}
\address{Current address: Faculty of Mathematics, Al. I. Cuza University of Iasi, Blvd. Carol I, no. 11, 700506 Iasi, Romania} \email{aykutkayhan@maltepe.edu.tr}
	

\thanks{The author was supported by the Scientific and Technological Research Council of T\"{u}rkiye (T\"{U}B\.{I}TAK), under the 2219-International Post-Doctoral Research Fellowship Programme (grant no: 1059B192300478). The opinions and views expressed herein are those of the authors and do not reflect those of T\"{U}B\.{I}TAK}

\begin{document}
	\maketitle
	
	\begin{abstract}
In this study, we investigate the intrinsic properties of compact biconservative hypersurfaces in space forms. In this framework, we establish rigidity results without imposing the assumption of constant scalar curvature. Furthermore, we present an additional result that does not require any assumptions on the sectional curvature. The key tool in our approach is the introduction of a novel divergence-free tensor, which enables us to derive these results without the usual curvature assumptions.
	\end{abstract}
	
	\section{Introduction}
	\begin{sloppypar}%To prevent line overflow
 Let $N^{m+1}(c)$ be an $(m+1)-$dimensional Riemannian manifold with constant sectional curvature $c$, we also call it a space form such that $N^{m+1}(1)=\mathbb{S}^{m+1}$, i.e. $(m+1)-$dimensional Euclidean sphere $,\ N^{m+1}(0)=\mathbb{R}^{m+1}$, i.e. $(m+1)-$dimensional Euclidean space,$\ N^{m+1}(-1)=\mathbb{H}^{m+1}$, i.e. $(m+1)-$dimensional hiperbolic space.
	\end{sloppypar}

	 Let $M^m$ be an $m-$ dimensional hypersurface in space form $N^{m+1}(c)$. The investigation of curvature characteristics of compact hypersurfaces remains a central and intriguing subject in the field. In 1977, S.Y. Cheng and S.T. Yau \cite{ChengYau1977} studied compact hypersurfaces with constant scalar curvature in space form $N^{m+1}(c)$. They proved that Let \(M\) be an \(m\)-dimensional compact hypersurface with constant scalar curvature \(m(m-1)R\). If \(R \geq c\) and the sectional curvatures of \(M\) are non negative, then \(M\) is isometric to either the totally umbilical hypersurface \(\mathbb{S}^m(r)\) or the Riemannian product \(\mathbb{S}^k(r_1) \times \mathbb{S}^{n-k}(r_2)\) for some \(1 \leq k \leq m-1\), where \(\mathbb{S}^k(r)\) denotes the \(k\)-dimensional sphere of radius \(r\). To prove this theorem, they introduced a differentiable operator $\square$, which was later named after them, such that
	 $$ \square\alpha=\langle mf \id-A,\Hess\alpha\rangle,$$
	 where $\alpha\in C^2(M)$, $f$ is the mean curvature function, and  $\id$ and $A$ denote the identity and shape operators of $M$, respectively. 
	  Importantly, when this operator is associated to a divergence-free, symmetric and $(1,1)$ tensor on $M$, it becomes self-adjoint, a property that significantly enhances its usefulness. In fact, the main strategy of this paper is to effectively utilize the Cheng–Yau operator after introducing such a tensor field. In this context, we provide a detailed discussion of this operator in the Preliminaries section.
	  
	 A natural question arises as to whether the condition of non-negative sectional curvature can be relaxed. In 1996, Li answered this question affirmatively by using a similar argument \cite{haizhong1996hypersurfaces}. He showed that let $M$ be an $m-$dimesional $(m\geq 3)$ compact hypersurface with constant scalar curvature $m(m-1)R$ in $\mathbb{S}^{m+1}$. If \(R \geq 1\) and $|A|^2\leq (m-1)\frac{m(R-1)+2}{m-2}+\frac{m-2}{m(R-1)+2}$, then $M$ is either totally umbilical or the Riemmanian product $\mathbb{S}^1(\sqrt{1-a^2})\times\mathbb{S}^{m-1}(a)$ with $a^2=\frac{m-2}{mR}\leq \frac{m-2}{m}$. Also In Li’s theorem, the Chen–Yau operator $\square$, constant scalar curvature, and the condition $R\geq 1$ are essential assumptions. it is natural to ask whether the condition of constant scalar curvature in these results is necessary? In 2003, the first answer to this question came from Q.-M. Cheng \cite{QMcheng2003compact}. He observed that some Riemannian products are not covered by the aforementioned results such that given $0 < a < 1$, by considering the standard immersions 
	 $
	 S^{n-1}(a) \subset \mathbb{R}^m, \quad S^1(\sqrt{1 - a^2}) \hookrightarrow \mathbb{R}^2,
	 $
	 and taking the Riemannian product immersion
	 $
	 S^1(\sqrt{1 - a^2}) \times S^{m-1}(a) \hookrightarrow \mathbb{R}^2 \times \mathbb{R}^m,
	 $
	 one obtains a compact hypersurface
	 $
	 S^1(\sqrt{1 - a^2}) \times S^{m-1}(a)
	 $
	 in $\mathbb{S}^{m+1}(1)$ with constant scalar curvature
	 $
	m(m - 1)R, \quad \text{where } R = \frac{m - 2}{ma^2} > 1 - \frac{2}{m}.
	 $	Motivated by this, Q.-M. Cheng conducted further studies in which he obtained results without assuming any conditions on the scalar curvature. However, as he himself noted, the problem appeared to be very difficult,  he tried to solve it under certain topological and additional geometric conditions \cite{QMcheng2003compact},\cite{QMcheng2005compact}. Indeed, due to the highly technical computations involved, one inevitably encounters equations that cannot be resolved without imposing additional conditions on the hypersurface. For instance, in this study, the term $\langle A(\grad f),\grad f\rangle$ appeared frequently throughout the calculations. At this point, we overcame the difficulty by making use of the notion of biconservative submanifolds which have gained significant attention in the last few years due to their intriguing properties. In 2014, the notion of biconservative manifold was introduced by R. Caddeo et. al.\cite{caddeo2014surfaces} and has subsequently developed into an active and growing area of research. If we focus only on the studies about biconservative submanifolds in space forms, papers like \cite{forliterature0hasanis1995hypersurfaces,forliterature0turgay2015h,forliterature1montaldo2016biconservative,forlitre2turmontaldo2016proper,forlit3fetcu2015cmc,fetcu2021bochner} can give readers a good (though not complete) idea. In fact,  biconservative submanifolds come from the biharmonic submanifolds which generalizes the consept of minimals submanifolds and they are isometric immersion $\varphi: M^m\to N^n$ satisfying the biharmonic equation
	$$
	\tau_2 (\varphi) = - \Delta^\varphi \tau (\varphi) - \trace R^N (d\varphi (\cdot), \tau (\varphi)) d \varphi (\cdot) = 0,
	$$
	where $\Delta^\varphi$ is the rough Laplacian acting on sections of the pull-back bundle $\varphi ^{-1} \left ( TN^n \right )$, $R^N$ is the curvature tensor field on $N^n$, $\tau (\varphi) = m H$ is the tension field associated to $\varphi$ and $H$ is the mean curvature vector field. The biharmonic equation decomposes into the normal and tangent parts. The biconservative submanifolds are characterized by
	\begin{equation} \label{eq:tau2tangentVanishing}
		(\tau_2 (\varphi))^\top = 0.
	\end{equation}
	In the case of biconservative hypersurfaces in space forms, i.e. $\varphi:M^m\to N^{m+1}(c)$, then biconservativity condition corresponds to the following
		\begin{equation}\label{eq:BiconservativityCharacterization}
		A(\grad f) = - \frac m 2 f \grad f.
	\end{equation}
	It is seen that every hypersurface with constant mean curvature (CMC) are trivially is biconservative. Therefore, the main interest lies in the investigation of non-CMC biconservative hypersurfaces. For recent developments and detailed discussions on this subject, the reader is referred to \cite{recentchen2024}, \cite{contemporary2022}.
	
	In this paper we present two rigidity results for compact biconservative hypersurfaces in space forms: the first does not require the assumption of constant scalar curvature, and the second does not impose any condition on the sectional curvature.
			\begin{theorem}\label{Theorem 1}
		Let $\varphi:M^m\to N^{m+1}(c)$ be a compact non minimal biconservative hypersurface in space form $ N^{m+1}(c)$ with non negative sectional curvature. If $|A|^2\leq\frac{m^2f^2}{6}$ then $\nabla A=0$ and $\varphi(M)$ is one of the following hypersurfaces
		\begin{itemize}
			\item[(a)] The Euclidean hypersphere $S^m(r)$ of radius $r > 0$, if $c \in \{-1, 0\}$, i.e. $N$ is either the hyperbolic space $\mathbb{H}^{m+1}$ or the Euclidean space $\mathbb{R}^{m+1}$;
			
			\item[(b)] Either the small hypersphere $S^m(r)$, $r \in (0,1)$, or the standard product $S^{m_1}(r_1) \times S^{m_2}(r_2)$, where $r_1^2 + r_2^2 = 1$, $m_1 + m_2 = m$, and $r_1 > \sqrt{1/m}$, if $c = 1$, i.e. $N$ is the unit Euclidean sphere $S^{m+1}$.
		\end{itemize}
	\end{theorem}
	
	\begin{theorem}\label{Theorem 2}
Let $\varphi:M^m\to \mathbb{R}^{m+1}$ be a compact non minimal biconservative hypersurface. If $m\geq 7$ and $|A|^2\leq \frac{m^2f^2}{m-1}$   then $\nabla A=0$ and $\varphi(M)$ is congruent to the hypersphere $\mathbb{S}^m(r)$ of radius $r$.
	\end{theorem}
		We would like to highlight the novelty of our approach: we introduce a divergence-free, symmetric $(1,1)-$tensor field (see Lemma \ref{div-free lemma}) that is valid for any hypersurface in space forms and establish an integral identity using biconservativity (see Lemma \ref{integral equlities}). By applying the Cheng–Yau operator, we obtain an integral equality (see \eqref{huge int equality to solve}), which we then analyze. We hope that this method can be further developed to yield more results for such hypersurfaces without assuming the condition of constant scalar curvature.
	\section{Preliminaries}
	
Throughout this paper, all manifolds are assumed to be connected, oriented, and of dimension at least two. Unless otherwise specified, the Riemannian metric on a given manifold is denoted by $\langle \cdot, \cdot \rangle$, or omitted when clear from context. The Levi-Civita connection of a Riemannian manifold $M$ is denoted by $\nabla$.

The rough Laplacian acting on sections of the pullback bundle $\varphi^{-1}(TN)$ is defined by
$$
\Delta^\varphi = - \trace \left ( \nabla ^\varphi \nabla ^\varphi - \nabla ^\varphi _\nabla \right ),
$$
where $\nabla^\varphi$ denotes the induced connection on the pullback bundle. The curvature tensor field is given by
$$
R (X, Y) Z = \left [ \nabla _X, \nabla _Y \right ] Z - \nabla _{[X, Y]} Z.
$$
Given a hypersurface $\varphi : M^m \to N^{m+1}$, the mean curvature function is defined by $f =\frac 1m \trace A$, where $A = A_\eta$ denotes the shape operator associated with a unit normal vector field $\eta$.

 We recall some fundamental results and formulas related to hypersurfaces in space forms, as well as properties of the Cheng–Yau operator. In particular, we present the Gauss and Codazzi equations for a hypersurface $M^m$ immersed in a space form $N^{m+1}(c)$. 
	\begin{itemize}
		\item The Gauss Equation is
		\begin{equation}\label{GaussEquation}
			R(X, Y) Z = c \left ( \left \langle Y, Z \right \rangle X - \left \langle X, Z \right \rangle Y \right ) + \left \langle A(Y), Z \right \rangle A(X) - \left \langle A(X), Z \right \rangle A(Y)
		\end{equation}
		for any $X, Y, Z \in C(TM)$.
		
		\item The Codazzi Equation is
		\begin{equation*} \label{CodazziEquation}
			\left ( \nabla _X A\right ) (Y) = \left ( \nabla _Y A \right ) (X)
		\end{equation*}
		for any $X, Y \in C(TM)$.
	\end{itemize}
	
	The scalar curvature $m(m-1)R$ of $M$ is expressed as 
	\begin{equation} \label{eq:formulaScalarCurvature}
		m (m - 1) (R-1) = m^2 f^2 - |A|^2
	\end{equation}	
	
Now, we recall several well-known properties of the shape operator of hypersurfaces in space forms, which will be used in the subsequent analysis
	
	\begin{lemma}\label{symmetric of A}
		Let $\varphi : M^m \to  N^{m+1}(c)$ be hypersurface in a space form. Then
		\begin{enumerate}[label = \alph*)]
			\item $(\nabla A) (\cdot, \cdot)$ is symmetric,
			\item $\left \langle (\nabla A) (\cdot, \cdot), (\cdot) \right \rangle$ is totally symmetric,
			\item $\trace (\nabla A) (\cdot, \cdot) = m \grad f$.
		\end{enumerate}
	\end{lemma}
	\begin{remark}
		Notice that Lemma \ref{symmetric of A} allows us to write 
		 $$\left \langle (\nabla A) (X, AY), Z \right \rangle = \left \langle (\nabla A) (X, Z), AY \right \rangle=\left \langle (\nabla A) (AY, X), Z \right \rangle $$
		 for any $X, Y, Z \in C(TM)$. It is important to emphasize this observation, as the further computations fundamentally rely on this approach (see Lemma \ref{main lemma1}).
	\end{remark}
	Let $\varphi : M^m \to N^{m+1} (c)$ be a hypersurface in a space form. We denote by $\{ \lambda_i \}_{i \in \overline {1, m}}$ the principal curvatures of $M$, that is the eigenvalue functions of the shape operator. While these functions are continuous in general, they may fail to be smooth on $M$. To study the structure more effectively, we consider the subset $M_A\subset M$ consisting of all points where the number of distinct principal curvatures remains locally constant. It is known that $M_A$ is open and dense in $M$. On a connected component of $M_A$, which is an open subset of $M$, the principal curvatures are smooth functions and there exists a local orthonormal frame field $\{E_i\}_{i \in \overline {1, m}}$ such that $A(E_i) = \lambda_i E_i$, for any $i \in \overline {1, m}$. When we consider the distinct principal curvatures of $A$, we denote their multiplicities by $m_1, \ldots, m_\ell$, where $\ell$ is the number of distinct principal curvatures.
	
	In the context of investigating the geometric properties of biconservative hypersurfaces, we make use of the following identity, which is valid for any hypersurface 	$\varphi : M^m \to N^{m+1} (c)$  in a space form.
	\begin{equation}\label{eq:GeneralFormulaSubmanifolds}
-\frac 12 |A|^2=|\nabla A|^2+\langle A,\Hess mf\rangle+\frac 1 2 \sum _{i, j = 1} ^m (\lambda_i - \lambda_j)^2 R_{ijij}
	\end{equation}
	In fact, this formula can be derived as a direct application of the Cheng–Yau formula
	\begin{equation} \label{eq:GeneralEquationYau}
	-	\frac 1 2 \Delta |\varPhi|^2 = |\nabla \varPhi|^2 + \langle \varPhi, \Hess (\trace \varPhi) \rangle + \frac 1 2 \sum _{i, j = 1} ^m (\mu_i - \mu_j)^2 R_{ijij},
	\end{equation}
	where $\varPhi$ is a symmetric $(1, 1)$ tensor field on an arbitrary manifold $M$ which satisfies the Codazzi equation, i.e. $(\nabla \varPhi) (X, Y) = (\nabla \varPhi) (Y, X)$, and $\mu_i$'s are the eigenvalues of $\varPhi$, for more see \cite{ChengYau1977}.
	
	Moereover, in the same work \cite{ChengYau1977} another important tool was introduced: the Cheng-Yau operator $\square$ associated to a symmetric $(1, 1)$ tensor field $\varPhi$. For any function $\gamma \in C^2 (M)$, $\square \gamma$ is defined by
	\begin{equation*}
		\square \gamma = \langle \varPhi, \Hess \gamma \rangle.
	\end{equation*} 
    In the case of  $\varPhi$ is a divergence-free tensor field defined on a compact manifold, then $\square$ is self-adjoint with respect to the $L^2$ inner product, i.e.
	\begin{equation*}
		\int _M \gamma (\square \theta) \ v_g = \int _M \theta (\square \gamma) \ v_g.
	\end{equation*}
	A direct consequence is that on compact manifolds 
	\begin{equation} \label{eq:consequenceChengYauSelfAdjoint}
		\int _M \square \gamma \ v_g = 0,
	\end{equation}
	for any divergence-free $(1, 1)$ tensor $\varPhi$.
	
	Now, we recall a few things about the stress-bienergy tensor $S_2$. Let $\phi :(M^m,g) \to \left ( N^n, \tilde{g} \right )$ be a smooth map, where $\tilde{g}$ is a Riemannian metric on $N$. Assume that $M$ is compact and on the set of all Riemannian metrics $g$ defined on $M$, one can define a new functional 
	\begin{equation*}
		\tilde E_2 (g) = \frac 1 2 \int _M \tilde{g}( \tau (\phi), \tau (\phi) )\ v_g,
	\end{equation*}
	where $\tau (\phi) = \trace_g \nabla d\phi$. The critical points of this functional are characterized by the vanishing of the stress-bienergy tensor $S_2$, see \cite{LoubeauMontaldoOniciuc2008}, where
	\begin{equation*}
		S_2 (X, Y) = \frac 12 |\tau (\phi)|^2 \langle X, Y \rangle + \langle d \phi, \nabla \tau (\phi) \rangle \langle X, Y \rangle - \langle d\phi (X), \nabla _Y \tau (\phi) \rangle - \langle d\phi (Y), \nabla _X \tau (\phi) \rangle,
	\end{equation*}
	for any $X, Y \in C(TM)$. The tensor field $S_2$ satisfies
	\begin{equation*}
		\Div S_2 = - \langle \tau_2 (\phi), d\phi \rangle,
	\end{equation*}
	see \cite{Jiang1986-2}.
	
	The biconservative submanifolds are defined by $\Div S_2 = 0$. In the case of hypersurfaces, the stress-bienergy tensor is given by
	\begin{equation*}
		S_2 = - \frac {m^2} 2 f^2 \id + 2mf A.
	\end{equation*}
	
	The following result, due to Ş. Andronic and the present author \cite{andronic2025}, provided a new characterization of biconservative hypersurfaces in space forms.
	
	\begin{lemma} \label{f2A}
	Let $M^m$ be a hypersurface in a space form $N^{m+1} (c)$. Then $\Div S_2 = 0$ if and only if $\Div \left ( f^2 A \right ) = 0$.
	\end{lemma}
In the same work \cite{andronic2025}, the authors also established the following integral identity for such  hypersurfaces as a consequence of Lemma \ref{f2A}  by using Cheng-Yau operator $\square$ associated to tensor field $f^2A$, i.e.
	\begin{equation}\label{int equ from firt paper}
		-\frac 12\int_M f^2\Delta|A|^2 v_g=\int_M f^2\biggl\{|\nabla A|^2+\frac 12\sum_i(\lambda_i-\lambda_j)^2R_{ijij}\biggr\}\ v_g.
	\end{equation}
	
	\section{Biconservative hypersurfaces without the assumption of constant scalar curvature} \label{sec:mLess10}
	First, we recall some examples of divergence-free $(1,1)$-tensors. Let $\Ric$ denote the Ricci tensor of a Riemannian manifold $M$. Then the tensor
	$$	T_1 = \tfrac{1}{2} m(m-1)R\ \id - \Ric$$
	is divergence-free.
	
	Another example is given by
	$$
	T_2 = (\trace S)\id - S,
	$$
	where $S$ is a symmetric $(1,1)$-tensor that satisfies the Codazzi equation.
	
	In the special case where $M$ is a hypersurface in a space form, $T_2$ coincides with
	\[
	T_2 = mf\, \mathrm{Id} - A,
	\]
	where $f$ denotes the mean curvature function and $A$ is the shape operator. Interestingly, when $M$ is biconservative hypersurface in a space form then $$T_3=f^2A$$ is divergence-free. However, these tensors $T_1,T_2$ and $T_3$ are not sufficient to relax the condition of constant scalar curvature.

	The strategy of this section is  to define a new divergence-free tensor and then to apply the Cheng–Yau operator associated with it. By taking advantage of the compactness of the hypersurface, the resulting integral expression vanishes, which leads to a nontrivial and extended integral equality (refer to  equation \eqref{huge int equality to solve}). The arguments required to analyze  this integral equality will be presented in the form of lemmas.
	
	The first lemma constitutes the main foundation of this study, and its results are stated in general for arbitrary dimensions $n$. In this paper, the cases $n = 2$ and $n = 3$ will play a central role in the resolution of the fundamental integral equality.
	
	\begin{lemma}\label{main lemma1}
	Let $\varphi : M^m \to  N^{m+1}(c)$ be hypersurface in a space form. Then
	\begin{enumerate}
		\item \label{grad trAn}$ \frac 1n \grad \trace A^n=\trace (\nabla A)(\cdot,A^{n-1}\cdot)$
		\item\label{nabla An}$(\nabla A^n)(X,Y)=(\nabla A)(X,A^{n-1}Y)+A((\nabla A^{n-1})(X,Y))$
	\end{enumerate}

where $A^n=\underbrace{A A \cdots A}_{n\text{ times}}$ and $X,Y$ tangent to $M$.
	\end{lemma}
	\begin{proof}
		Let $\{E_i\}_{i \in \overline {1, m}}$ be a local orthonormal frame field  on $M$  such that $A(E_i) = \lambda_i E_i$, for any $i \in \overline {1, m}$. Then we have
		
		\eqref{grad trAn}$\Leftrightarrow$
		\begin{eqnarray*}
			\trace(\nabla A)(\cdot,A^{n-1}\cdot)&=&\sum_{i=1}^m(\nabla A)(E_i,A^{n-1}E_i)\\
			&=&\sum _{i, j = 1} ^m\langle (\nabla A)(E_i,A^{n-1}E_i),E_j\rangle E_j\\
			&=&\sum _{i, j = 1} ^m\langle (\nabla A)(E_j,E_i),A^{n-1}E_i  \rangle E_j\\
			&=&\sum _{i, j = 1} ^m\langle \nabla_{E_j}(AE_i)-A(\nabla_{E_j}E_i),A^{n-1}E_i\rangle E_i\\
			&=&\sum _{i, j = 1} ^m\langle \nabla_{E_j}(AE_i),A^{n-1}E_i \rangle E_j\\
			&=&\sum _{i, j = 1} ^m\lambda_i^{n-1}{E_j}\langle AE_i,E_i  \rangle E_j\\
			&=&\sum _{i, j = 1} ^m\lambda_i^{n-1}E_j(\lambda_i) E_j\\
			&=& \frac 1n \grad\trace A^n.
		\end{eqnarray*} 
		
		 Now, for the second,  note that $(\nabla_XA)(A^{n-1}Y)=\nabla_X(A(A^{n-1}Y))-A(\nabla_X(A^{n-1}Y))$. Then,
		 
		 \eqref{nabla An}$\Leftrightarrow$
		\begin{eqnarray*}
(\nabla A^n)(X,Y)&=&\nabla_X(A^nY)-A^n(\nabla_XY)\\
				&=&\nabla_X(A(A^{n-1}Y))-A(A^{n-1}(\nabla_XY))+ \bigg(A(\nabla_XA^{n-1}(Y))-A(\nabla_X(A^{n-1}Y))\bigg)\\
				&=&\nabla_X(A(A^{n-1}Y))-A(\nabla_X(A^{n-1}Y))+A(\nabla_XA^{n-1}(Y))-A(A^{n-1}(\nabla_XY))\\
				&=&(\nabla A)(X,A^{n-1}Y)+A((\nabla A^{n-1})(X,Y)).
		\end{eqnarray*}
		This technique of the last computation is essentially due to Ş. Andronic.
	\end{proof}
	From now on  we  denote $\trace A^2=\vert A\vert^2$. Then it is obvious that we have
	\begin{eqnarray}
	\label{gradnorm A}	\frac 12 \grad\vert A\vert^2&=&\trace(\nabla A)(\cdot,A\cdot),\\
	\frac 13\grad\trace A^3&=&\trace (\nabla A)(\cdot,A^2\cdot).
	\end{eqnarray}
	In a Riemannain manifold, divergence of a symmetric and $(1,1)-$tensor field $T$ is given by
	$$\Div T=\trace(\nabla T)(\cdot,\cdot).$$
So, because of $A^2$ and $A^3$ are symmetric and $(1,1)$ tensor, we can give the following lemma which is direct consequence of Lemma \ref{main lemma1}.
	\begin{lemma}\label{divA2A3}
Let $\varphi:M\to N^{m+1}(c)$ be a  hypersurface in space form $ N^{m+1}(c)$. Then the following identities hold:
\begin{eqnarray}
\label{divA2}\Div A^2&=&\frac 12\grad\vert A\vert^2+mA(\grad f)\\
\label{divA3}\Div A^3&=&\frac 13\grad \trace A^3+\frac 12 A(\grad\vert A\vert^2)+mA^2(\grad f)
\end{eqnarray}
	\end{lemma}
	\begin{proof}
		We have
		$$\Div A^2=\trace(\nabla A^2)(\cdot,\cdot)\ \mbox{  and  }\Div A^3=\trace(\nabla A^3)(\cdot,\cdot).$$
		Putting $X=Y=E_i$ and $n=2$ in the equation \eqref{nabla An} in Lemma \ref{main lemma1}, we have
		\begin{eqnarray*}
\Div A^2&=&\trace(\nabla A^2)(\cdot,\cdot)\\
		&=& \trace(\nabla A)(\cdot,A\cdot)+A(\trace(\nabla A)(\cdot,\cdot))\\
		&=&\frac 12 \grad\vert A\vert^2+A(\Div A)\\
		&=&\frac 12 \grad\vert A\vert^2+mA(\grad f).
		\end{eqnarray*}
Similarly, putting $X=Y=E_i$ and $n=3$ in the equation \eqref{nabla An} in Lemma \ref{main lemma1}, we get
\begin{eqnarray*}
\Div A^3&=&\trace(\nabla A^3)(\cdot,\cdot)\\
		&=&\trace(\nabla A)(\cdot,A^2\cdot)+A(\trace(\nabla A^2)(\cdot,\cdot))\\
		&=&\frac 13 \grad\trace A^3+A(\Div A^2)\\
		&=&\frac 13 \grad\trace A^3+\frac 12 A(\grad\vert A\vert^2)+mA^2(\grad f).
\end{eqnarray*}
	
	\end{proof}
	\begin{lemma}\label{integral equlities}
Let $\varphi:M\to N^{m+1}(c)$ be a compact biconservative hypersurface in space form $ N^{m+1}(c)$. Then we have the following integral identities,
\begin{eqnarray}
\label{intf2A2}\frac 12 \int_M f^2\Delta \vert A\vert^2\ v_g&=&\int_M \frac{m^2f^2}{2}\vert\grad f\vert^2-2f\langle A^2,\Hess f \rangle\ v_g\\
\label{intfA3}-\frac 13 \int_M f\Delta\trace A^3\ v_g&=&\int_M \frac{m^3f^2}{8}\vert\grad f\vert^2+\frac m2 f\langle A^2,\Hess f\rangle+\langle \Hess f,A^3 \rangle\ v_g.
\end{eqnarray}
	\end{lemma}
	\begin{proof}
	Let $\{E_i\}_{i \in \overline {1, m}}$ be a local orthonormal frame field  on $M$  which diagonalizes the shape operator.	To prove first one, we would like to calculate the term $\langle \grad f^2,\grad\vert A\vert^2 \rangle$. In this manner, using \eqref{gradnorm A} we get
		\begin{align*}
			\frac 12\langle \grad f^2,\grad\vert A\vert^2 \rangle&=\sum_{i=1}^m\langle \grad f^2,(\nabla A)(E_i,AE_i) \rangle\\
			&=\sum_{i=1}^m \ang{(\nabla A)(E_i,\grad f^2)}{AE_i}\\
			&=\sum_{i=1}^m\ang{\nabla _{E_i}(A(\grad f^2))-A(\nabla_{E_i}\grad f^2)}{AE_i}\\
			&=\sum_{i=1}^m \ang{\nabla_{E_i}(2f(\frac{-mf}{2})\grad f)}{AE_i}-\ang{\nabla_{E_i}\grad f^2}{A^2E_i}\\
			&= -m \sum_{i=1}^m \biggl\{
			\left\langle E_i(f^2) \, \grad f + f^2 \nabla_{E_i} \grad f, \, A E_i \right\rangle 
			\biggr\} \\
			&\quad - 2 \sum_{i=1}^m \biggl\{
			\left\langle E_i(f) \, \grad f + f \nabla_{E_i} \grad f, \, A E_i \right\rangle
			\biggr\}\\
			&=-m\ang{A(\grad f^2)}{\grad f}-f^2\ang{A}{\Hess mf}-2\biggl\{\vert A(\grad f)\vert^2+f\ang{A^2}{\Hess f}\biggr\}.
		\end{align*}
		From which we get
		\begin{equation}\label{first int equality proof}
\frac 12\langle \grad f^2,\grad\vert A\vert^2 \rangle=\frac{m^2f^2}{2}\vert\grad f\vert^2-f^2\ang{A}{\Hess mf}-2f\ang{A^2}{\Hess f}.
		\end{equation}
	Note that $\Div f^2A=0$ since $M$ is biconservative. Then,  Applying Cheng-Yau square operator $\square$ by taking into account that $M$ is compact we obtain $\int_M f^2\ang{A}{\Hess mf}=0.$ Moreover, using integration by part one can say $\int_M \langle \grad f^2,\grad\vert A\vert^2 \rangle=\int_M f^2\Delta\vert A\vert^2 $ since compactness of $M$. So, integrating \eqref{first int equality proof} over $M$ we obtain \eqref{intf2A2}.
	
	Now, for the second, we shall compute the term $\ang{\grad f}{\grad\trace A^3}$. We have
	\begin{eqnarray*}
\frac 13 \ang{\grad f}{\grad\trace A^3}&=&\sum_{i=1}^m \ang{\grad f}{(\nabla A)(E_i,A^2E_i)}\\
									   &=&\sum_{i=1}^m \ang{(\nabla A)(E_i,\grad f^2)}{A^2E_i}\\
									   &=&\sum_{i=1}^m \ang{\nabla_{E_i}(A(\grad f))-A(\nabla_{E_i}\grad f)}{A^2E_i}\\
									   &=&\sum_{i=1}^m \ang{\nabla_{E_i}(-\frac{mf}{2}\grad f)}{A^2E_i}-\ang{\nabla_{E_i}\grad f}{A^3E_i}\\
									   &=&\sum_{i=1}^m -\frac m2\biggl\{\vert A(\grad f)\vert^2+f\ang{A^2}{\Hess f}\biggr\}-\ang{A^3}{\Hess f}
	\end{eqnarray*}
	From which we get
	\begin{equation}\label{second int equality in proof}
		\frac 13 \ang{\grad f}{\grad\trace A^3}=-\frac{m^3f^2}{8}\vert\grad f\vert^2-\frac m2 f\langle A^2,\Hess f\rangle-\langle \Hess f,A^3 \rangle.
	\end{equation}
	İntegrating \eqref{second int equality in proof}  we obtain \eqref{intfA3}.
	\end{proof}
	Before proceed, notice that combining the  integral identities given in \eqref{int equ from firt paper} and \eqref{intf2A2}, we derive the following identity, which will be useful in later computations.
	\begin{equation}\label{fA2hessf}
\int_M 2f\langle A^2,\Hess f\rangle=\int_M \frac{m^2f^2}{2}|\grad f|^2+f^2\biggl\{|\nabla A|^2+\frac 12\sum_{i,j=1}^m(\lambda_i-\lambda_j)^2R_{ijij}\biggr\}\ v_g.
	\end{equation}
	
	In the next lemma, we introduce a divergence-free (1,1)-tensor  that plays a key role in this work.
\begin{lemma}\label{div-free lemma}
		Let \(\varphi : M^m \to N^{m+1}(c)\) be a hypersurface in \(N^{m+1}\).  
		Let \(\phi\) be a \((1,1)\) tensor given by
		\[
		\phi = \psi_3 \, \id - \psi_2 A + mf A^2 - A^3,
		\]
		where
		\[
		\psi_3 = \frac{1}{6} m^3 f^3 + \frac{1}{3} \trace A^3 - \frac{1}{2} mf \lvert A \rvert^2, \quad
		\psi_2 = \frac{1}{2} \left( m^2 f^2 - \lvert A \rvert^2 \right).
		\]
		Then \(\mathrm{div} \, \phi = 0\).
	\end{lemma}
	\begin{proof}
	To present the proof, we shall proceed with a direct computation.
		\begin{equation}\label{divphifirst}
\Div\phi=\grad \psi_3-A(\grad \psi_2)-\psi_2m\grad f+m\left(A^2(\grad f)+f\Div A^2\right)-\Div A^3\\
		\end{equation}
		A term-by-term analysis leads to the following:
		\begin{subequations}\label{aburcubur}
		\begin{align}
			\grad \psi_3&=\frac{m^2}{6} 3f^2\grad f+\frac 13\grad\trace A^3-\frac m2\biggl(\vert A\vert^2\grad f+f\grad\vert A\vert^2\biggr)\\
			A(\grad \psi_2)&=\frac 12 \biggl(m^22fA(\grad f)-A(\grad\vert A\vert^2)\biggr)
		\end{align}
	\end{subequations}
	Substituting \eqref{aburcubur} into \eqref{divphifirst} and taking into account Lemma \ref{main lemma1}, we get
\begin{align*}
	\Div\phi&=\frac{m^3}{6}3f^2\grad f+\frac 13\grad\trace A^3-\frac m2\biggl(\vert A\vert^2\grad f+f\grad\vert A\vert^2\biggr)\\
			&\quad-m^2fA(\grad f)-\frac 12A(\grad\vert A\vert^2)-\frac 12(m^2f^2-\vert A\vert^2)m\grad f\\
			&\quad +mA^2(\grad f)+mf\biggl(\frac 12\grad\vert A\vert^2+A(m\grad f)\biggr)\\
			&\quad-\biggl(\frac 13\grad\trace A^3+\frac 12A(\grad\vert A\vert^2)+A^2(m\grad f)\biggr).
\end{align*} 
After simplifyig we get $\Div \phi=0$.
	\end{proof}
	Before using the Cheng–Yau operator $\square$ associated to  $\phi$, the following lemma must be established.
	\begin{lemma}\label{14lemma}
Let $\varphi:M^m\to N^{m+1}$ be a hypersurface in Riemannian manifold $N$. Then 
\begin{equation}\label{1/4grad}
\frac 14\vert\grad\vert A\vert^2\vert^2\leq\vert A\vert^2\vert\nabla A\vert^2 
\end{equation}
	\end{lemma}
	\begin{proof}
Let $\{E_i\}_{i \in \overline {1, m}}$ be a local orthonormal frame field  on $M$  which diagonalizes the shape operator. We know from \eqref{gradnorm A} that
$$\frac 12\grad\vert A\vert^2=\sum_{i=1}^m(\nabla A)(E_i,AE_i)=\sum_{i,j=1}^m\ang{(\nabla A)(E_i,AE_i)}{E_j}E_j.$$
From which we get
\begin{align*}
\frac 14\vert\grad\vert A\vert^2\vert^2&=\biggl\vert\sum_{i,j=1}^m\langle (\nabla A)(E_i,AE_i),E_j\rangle E_j\biggr\vert^2\\
										&=\sum_{i,j=1}^m\langle (\nabla A)(E_i,E_j),AE_i \rangle^2\\
										&=\sum_{i,j,k=1}^m |(\nabla A)(E_i,E_j)|^2|A(E_i)|^2\\
										&\leq\sum_{i,j,k=1}^m |(\nabla A)(E_i,E_j)|^2|A(E_k)|^2 \\
										&= \vert\nabla A\vert^2\vert A\vert^2.
\end{align*}
	\end{proof}
	Now, since we have div-free $(1,1)$ tensor $\phi$ and $M$ is compact, applying the Cheng-Yau operator $\square$  associated to $\phi$, we have from \eqref{eq:consequenceChengYauSelfAdjoint} that
	$$\int_M\square f \ v_g=\int_M\langle \phi,\Hess f\rangle\ v_g=0$$
	that is
	\begin{equation}\label{main equ for CY}
		\int_M \langle \psi_3 \, \id - \psi_2 A + mf A^2 - A^3,\Hess f\rangle=0.
	\end{equation}
	To make the computations easier to follow, we proceed step by step. 
	 We will evaluate each term in the integrand separately.
	
	We start with
	\begin{align*}
		\int_M\ang{\psi_3\id}{\Hess f}&=-\int_M \psi_3\Delta f\\
									&=-\int_M\frac 16 m^3f^3\Delta f+\frac 13\trace A^3\Delta f-\frac 12mf\vert A\vert^2\Delta f\\
									&=-\int_M\frac 16 m^33f^2\vert\grad f\vert^2+\frac 13 f\Delta\trace A^3\\
									&\quad-\frac m2\langle |A|^2\grad f+f\grad\vert A\vert^2,\grad f \rangle.
	\end{align*}
	From which it follows that	
	\begin{equation}\label{firs term}
			\int_M\ang{\psi_3\id}{\Hess f}=-\int_M(\frac{m^3}{2}f^2-\frac m2 |A|^2)|\grad f|^2+\frac 13 f\Delta\trace A^3-\frac m4 f^2\Delta|A|^2.
	\end{equation}
	Now, followed by the term
	\begin{align*}
-\int_M \psi_2\langle A,\Hess f\rangle&=-\int_M \frac{1}{2m}(m^2f^2-|A|^2)\langle A,\Hess mf\rangle\\
									&=\int_M \frac{1}{2m}(m^2f^2-|A|^2)\biggl(\frac 12\Delta|A|^2+|\nabla A|+\frac 12\sum_i(\lambda_i-\lambda_j)^2R_{ijij}\biggr)
	\end{align*}
	From which we get
	\begin{equation}\label{second term}
-\int_M \psi_2\langle A,\Hess f\rangle=\int_M \frac m4 f^2\Delta|A|^2-\frac{1}{4m}|\grad|A|^2|^2+\frac{1}{2m}(m^2f^2-|A|^2)\biggr(|\nabla A|^2+\frac 12\sum_{i,j=1}^m(\lambda_i-\lambda_j)^2R_{ijij}\biggl)
	\end{equation}
	Substituting \eqref{firs term} and \eqref{second term} into \eqref{main equ for CY}, we obtain
	\begin{align}\label{huge int equality to solve}
		0&=\int_M-\frac m2(m^2f^2-|A|^2)|\grad f|^2-\frac 13 f\Delta\trace A^3+\frac m2 f^2\Delta|A|^2-\frac{1}{4m}|\grad|A|^2|^2\\
	\nonumber	&\quad+\frac{1}{2m}(m^2f^2-|A|^2)\biggl(|\nabla A|^2+\frac 12\sum_{i,j=1}^m(\lambda_i-\lambda_j)^2R_{ijij}\biggr)+mf\ang{A^2}{\Hess f}-\ang{A^3}{\Hess f}.
	\end{align}
	Referring to the integral identities established in Lemma \ref{integral equlities}, we are able to simplify a portion of the integral presented in \eqref{huge int equality to solve}. The computation proceeds as follows:
	\begin{align}\label{simplifying 1}
		\int_M-\frac 13 f\Delta\trace A^3+\frac m2 f^2\Delta|A|^2+mf\ang{A^2}{\Hess f}-\ang{A^3}{\Hess f}&=\int_M -\frac m2 f\langle A^2,\Hess f\rangle\\
	\nonumber																									&\quad +\frac{5m^3f^2}{8}|\grad f|^2.
	\end{align}
	We have from \eqref{fA2hessf} that
	\begin{equation}\label{simplifying 2}
\int_M -\frac m2 f\langle A^2,\Hess f\rangle+\frac{5m^3f^2}{8}|\grad f|^2=\int_M \frac{m^3f^2}{2}|\grad f|^2-\frac m4f^2\biggl(|\nabla A|^2+\frac 12\sum_{i,j=1}^m(\lambda_i-\lambda_j)^2R_{ijij}\biggr)
	\end{equation}
	On the other hand, we have from \eqref{1/4grad} that
	\begin{equation}\label{simplifying 3}
		-\frac{1}{4m}|\grad|A|^2|^2\geq -\frac 1m |\nabla A|^2|A|^2.
	\end{equation}
		Now, substituting \eqref{simplifying 1}, \eqref{simplifying 2} and \eqref{simplifying 3} into \eqref{huge int equality to solve}, we obtain
	\begin{equation}\label{huge int 2}
		0\geq \int_M \frac m2|A|^2|\grad f|^2+\big(\frac{mf^2}{4}-\frac{3}{2m}|A|^2\big)|\nabla A|^2+\biggl(\frac{mf^2}{4}-\frac{1}{2m}|A|^2\biggr)\frac 12\sum_{i,j=1}^m(\lambda_i-\lambda_j)^2R_{ijij}
	\end{equation}
	After establishing the main integral inequality above, we can give the proof of the main theorem of this work.

	\subsection{The proof of Theorem \ref{Theorem 1}}:

		\vspace*{1\baselineskip}
	Let $|A|^2\leq \frac{m^2f^2}{6}$. Then \eqref{huge int 2} becomes
		\begin{equation*}
0\geq\int_M \frac m2|A|^2|\grad f|^2+\frac 12 \frac{mf^2}{6}\sum_{i,j=1}^m(\lambda_i-\lambda_j)^2R_{ijij},
		\end{equation*}
		which leads to 
		$$f^2\sum_{i,j=1}^m(\lambda_i-\lambda_j)^2R_{ijij}=0$$
		since sectional curvature is non negative. Using again the same fact that,  we obtain
		\begin{equation}\label{f=0}
		f^2(\lambda_i-\lambda_j)^2R_{ijij}=0.
		\end{equation}
	
	 Our claim is to show that $\grad f=0.$ If $f=0$ then it obviously implies that the claim is true. Now, because $M$ is non minimal then there exists at least one point $p$ on $M$ such that $f(p)\neq 0.$ Then, it follows from \eqref{f=0} that  
	 \begin{equation*}
	 	(\lambda_i - \lambda_j)^2 R_{ijij} = 0, \quad \forall i, j \in \overline {1, m}.
	 \end{equation*}
	 at any point of an open neighbourhood $U$ of $p$. Since, from the Gauss equation, for any distinct $i, j \in \overline {1, m}$, we have $R_{ijij} = c + \lambda_i \lambda_j$, we deduce that
	 \begin{equation*}
	 	(\lambda_i - \lambda_j)^2 (c + \lambda_i \lambda_j) = 0, \quad \forall i, j \in \overline {1, m}.
	 \end{equation*}
	 In fact, on $M$, we obtain
	 \begin{equation} \label{eq:RelationPrincipalCurvatures}
	 	(\lambda_i - \lambda_j) (c + \lambda_i \lambda_j) = 0, \quad \forall i, j \in \overline {1, m}.
	 \end{equation}
	 The last relation implies that $M$ has at most two distinct principal curvatures at any point of $M$.
	 
	 Consider now the subset $M_A$ of all point in $M$ at which the number of distinct principal curvatures is locally constant. In the following we will show that $\grad f = 0$ on every connected component of $M_A$ and thus, from density, we will conclude that $\grad f = 0$ on $M$, i.e. $f$ is constant. Then, from \eqref{eq:GeneralFormulaSubmanifolds} and the fact that $M$ is compact, we get that $\nabla A = 0$.
	 
	 We choose an arbitrary connected component of $M_A$. Since $M$ has at most two distinct principal curvatures, on this component we have: either each of its points is umbilical, or each of its points has exactly two distinct principal curvatures. For simplicity, we denote by $M$ the chosen connected component.
	 
	 If $M$ is umbilical, then it is CMC.
	 
	 We suppose now that $M$ has exactly two distinct principal curvatures at any point. In this case, $A$ is locally diagonalizable with respect to a smooth orthonormal frame field $\{E_i\}_{i \in \overline {1, m}}$. Denote
	 \begin{equation*}
	 	\lambda_1 = \ldots = \lambda_{m_1} \quad \text{and} \quad \lambda_{m_1 + 1} = \ldots = \lambda_m.
	 \end{equation*}
	 
	 Assume that $\grad f \neq 0$ and we will obtain a contradiction. If necessary, we can restrict ourselves to an open subset of $M$, denoted again by $M$, such that $\grad f \neq 0$ at any point of $M$.
	 
	 Now, using \eqref{eq:BiconservativityCharacterization} we can assume that
	 \begin{equation*}
	 	\lambda_1 = - \frac m 2 f, \quad m_1 = 1 \quad \text{and} \quad E_1 = \frac {\grad f} {|\grad f|}
	 \end{equation*}
	 on $M$. Since $\trace A = mf$, we have
	 \begin{equation*}
	 	\lambda_2 = \frac {3m} {2(m - 1)} f.
	 \end{equation*}
	 Using \eqref{eq:RelationPrincipalCurvatures}, we obtain
	 \begin{equation*}
	 	0 = c + \lambda_1 \lambda_2 = c - \frac {3m^2} {4(m - 1)} f^2.
	 \end{equation*}
	 We wouldlike you to notice that this relation fails if $c \leq 0$. If $c > 0$, because of $f$ is smooth, we obtain that $f$ is constant on $M$, which contradicts $\grad f \neq 0$ at any point of $M$. 
	 
	 Now, from \eqref{eq:GeneralFormulaSubmanifolds} and the compactness of $M$  we obtain $\nabla A=0$. Then $M$ has at most two distinct principal curvatures which implies that either $M$ is totally umbilical or has exactly two distinct principal curvatures. In order to classify such a hypersurface we need to consider the cases $c=-1,\ c=0,\ c=1$.
	 
	 In the case of $c=-1$, $M$ must be totally umbilical due to its compactness, additionally, $\varphi(M)$ is a hypersphere $\mathbb{S}^m(r)$ of radius $r$.
	 
	 When $c=0$, again, $M$ must be totally umbilical due to its compactness and $\varphi(M)$ is a hypersphere $\mathbb{S}^m(r)$ of radius $r$.
	 
	 Now, we consider $c=1$. If $M$ is totally umbilical then $\varphi(M)$ is  a small hypersphere   $\mathbb{S}^m(r)$ of  $\mathbb{S}^{m+1}(r)$. If $M$ has two distinct principal curvature then  $\varphi(M)$ is the standart product $\mathbb{S}^1(r_1)\times\mathbb{S}^{m-1}(r_2)$, where $r_1^2+r_2^2=1$. For any $X=(X_1,X_2)\in C(T(\mathbb{S}^1(r_1)\times\mathbb{S}^{m-1}(r_2)))$, we have
	 $$AX=\biggr(-\frac{r_2}{r_1}X_1,\frac{r_1}{r_2}X_2\biggl)$$
	 and we obtain
	 $$\lambda_1=-\frac{r_2}{r_1}\mbox{  and  }\lambda_2=\frac{r_1}{r_2}.$$
	 From which we say
	 $$|A|^2=\biggl(\frac{r_2}{r_1}\biggr)^2+(m-1)\biggl(\frac{r_1}{r_2}\biggr)^2$$
	 and
	 $$m^2f^2=\biggl(\frac{r_2}{r_1}\biggr)^2-2(m-1)+(m-1)^2\biggl(\frac{r_1}{r_2}\biggr)^2.$$
	  we have from the hypothesis that
	 $$ \biggl(\frac{r_2}{r_1}\biggr)^2+(m-1)\biggl(\frac{r_1}{r_2}\biggr)^2\leq \frac 16 \biggl\{\biggl(\frac{r_2}{r_1}\biggr)^2-2(m-1)+(m-1)^2\biggl(\frac{r_1}{r_2}\biggr)^2\biggr\}.$$
	 It follows that
	 $$5\biggr(\frac{r_2}{r_1}\biggr)^2-(m-7)(m-1)\biggl(\frac{r_1}{r_2}\biggr)^2\leq 0$$
	 which leads to
	 $$5\biggr(\frac{r_2}{r_1}\biggr)^4\leq (m-7)(m-1)< (m-1)^2< 5(m-1)^2.$$
	 Then
	 \begin{equation}\label{eq:r1}
	 	\biggl(\frac{1}{r_1^2}-1\biggr)^2< (m-1)^2
	 \end{equation}
	 Notice that $\frac{1}{r_1^2} > 1$ since $r_1^2+r_2^2=1$. Taking into account this fact, solving the inequality \eqref{eq:r1} for $r_1$ one can get $r_1>\sqrt{1/m}.$
	  \hfill\qed
	  
		In Theorem \ref{Theorem 1}, we succeeded in relaxing the strong assumption of constant scalar curvature. The natural question that follows is whether the condition of non-negative sectional curvature can also be relaxed. Li answered this question affirmatively by employing Okumura's lemma, which is an algebraic result \cite{haizhong1996hypersurfaces}. Later, Andronic and the author further developed this result, leading to the conclusion stated in the following lemma (for detail see the proof of Theorem 3.6 in \cite{andronic2025}).
		\begin{lemma}\label{lemma: li adapted}\cite{andronic2025}
			Let $\varphi:M^m\to N^{m+1}(c)$ be a hypersurface in a space form. If $c+f^2\geq 0$ when $c<0$ and
			\begin{align}\label{eq: c case}
				|A|^2 \leq mc + \frac {m^3} {2(m-1)} f^2 - \frac {m(m-2)} {2(m-1)} \sqrt { m^2 f^4 + 4 (m-1) cf^2},						 
			\end{align}
			then
			$$\sum_{i,j=1}^m(\lambda_i-\lambda_j)^2R_{ijij}\geq 0$$
		\end{lemma}
	Now, we can give the proof of the second theorem of this work.
\subsection{The proof of Theorem \ref{Theorem 2}}:

	First notice that the hypothesis
	$$|A|^2\leq\frac{m^2f^2}{m-1}$$
	obviously implies that $\sum\limits_{i,j=1}^m(\lambda_i-\lambda_j)^2R_{ijij}\geq 0$  from Lemma \ref{lemma: li adapted}, since the ambient space is $\mathbb{R}^{m+1}$, i.e. $c=0$. Moreover, for $m\geq 7$, we have
	$$|A|^2\leq\frac{m^2f^2}{m-1}\leq \frac{m^2f^2}{6}.$$
		
	Thus, Theorem \ref{Theorem 1} completes the proof.\hfill\qed
	
	\textbf{Conflict of interest}
	
	The author declares that there is no conflict of interest.
	
	\textbf{Data Availability}
	
	No data was used for the research described in the article.
	\bibliographystyle{abbrv}
	\bibliography{Bibliography.bib}
\end{document}